\documentclass{article}
\usepackage{amssymb,amsfonts,amsmath,amsthm,amstext,amscd,latexsym,xy,color}
\input xy
\xyoption{all}

\newtheorem{theorem}{Theorem}[section]
\newtheorem{lemma}[theorem]{Lemma}
\newtheorem{proposition}[theorem]{Proposition}
\newtheorem{corollary}[theorem]{Corollary}
\newtheorem{definition}[theorem]{Definition}

\newtheorem{question}[theorem]{Question}
\newtheorem{principle}[theorem]{Principle}

\DeclareMathOperator{\Frob}{Frob}
\DeclareMathOperator{\tr}{tr}
\DeclareMathOperator{\Spec}{Spec}

\newcommand{\cA}{{\mathcal A}}
\newcommand{\cB}{{\mathcal B}}
\newcommand{\cC}{{\mathcal C}}

\newcommand{\cH}{{\mathcal H}}

\newcommand{\cL}{{\mathcal L}}

\newcommand{\cO}{{\mathcal O}}

\newcommand{\frg}{{\mathfrak g}}

\newcommand{\ffrm}{{\mathfrak m}}

\newcommand{\frp}{{\mathfrak p}}

\newcommand{\bbA}{{\mathbb A}}

\newcommand{\bbF}{{\mathbb F}}

\newcommand{\bbQ}{{\mathbb Q}}

\newcommand{\bbT}{{\mathbb T}}

\newcommand{\bbZ}{{\mathbb Z}}
\newcommand{\ord}{{\operatorname{ord}}}
\newcommand{\GL}{{\operatorname{GL}}}
\newcommand{\Bun}{{\operatorname{Bun}}}
\newcommand{\End}{{\operatorname{End}}}
\newcommand{\SL}{{\operatorname{SL}}}
\newcommand{\Art}{{\operatorname{Art}}}
\newcommand{\Gal}{{\operatorname{Gal}}}
\newcommand{\Aut}{{\operatorname{Aut}}}
\newcommand{\Lift}{{\operatorname{Lift}}}
\newcommand{\Def}{{\operatorname{Def}}}
\newcommand{\Sets}{{\operatorname{Sets}}}
\newcommand{\Cent}{{\operatorname{Cent}}}
\newcommand{\hG}{{\widehat{G}}}
\newcommand{\hT}{{\widehat{T}}}
\title{Potential automorphy of $\hG$-local systems}

\begin{document}
\author{Jack A. Thorne}
\maketitle

\begin{abstract}
Vincent Lafforgue has recently made a spectacular breakthrough in the setting of the global Langlands correspondence for global fields of positive characteristic, by constructing the `automorphic--to--Galois' direction of the correspondence for an arbitrary reductive group $G$. We discuss a result that starts with Lafforgue's work and proceeds in the opposite (`Galois--to--automorphic') direction. 
\end{abstract}

\vskip .5in

\section{Introduction}

Let $\bbF_q$ be a finite field, and let $G$ be a split reductive group over $\bbF_q$. Let $X$ be a smooth projective connected curve over $\bbF_q$, and let $K$ be its function field. Let $\bbA_K$ denote the ring of ad\`eles of $K$. Automorphic forms on $G$ are locally constant functions $f : G(\bbA_K) \to \overline{\bbQ}$ which are invariant under left translation by the discrete group $G(K) \subset G(\bbA_K)$.\footnote{This is not the full definition. In the rest of this article we consider only cuspidal automorphic forms, which are defined precisely below.}  The space of automorphic forms is a representation of $G(\bbA_K)$, and its irreducible constituents are called automorphic representations. 

Let $\ell \nmid q$ be a prime. According to the Langlands conjectures, any automorphic representation $\pi$ of $G(\bbA_K)$ should give rise to a continuous representation $\rho(\pi) : \pi_1^\text{\'et}(U) \to \hG(\overline{\bbQ}_\ell)$ (for some open subscheme $U \subset X$). This should be compatible with the (known) unramified local Langlands correspondence, which describes the pullback of $\rho(\pi)$ to $\pi_1^\text{\'et}(\bbF_{q_v})$ for every closed point $v : \Spec \bbF_{q_v} \hookrightarrow U$ in terms of the components $\pi_v$ of a factorization $\pi = \otimes'_v \pi_v$ into representations of the local groups $G(K_v)$.

Vincent Lafforgue has given an amazing construction of the representation $\rho(\pi)$, which crystallizes and removes many of the ambiguities in this picture in a beautiful way. We give a brief description of this work below. 

Our main goal in this article is to describe a work due to Gebhard B\"ockle, Michael Harris, Chandrashekhar Khare, and myself, where we establish a partial converse to this result \cite{Boc17}. We restrict to representations $\rho : \pi_1^\text{\'et}(X) \to \hG(\overline{\bbQ}_\ell)$ of Zariski dense image, and show that any such representation is \emph{potentially automorphic}, in the sense that there exists a Galois cover $Y \to X$ such that the pullback of $\rho$ to $\pi_1^\text{\'et}(Y)$ is contained in the image of Lafforgue's construction. We will guide the reader through the context surrounding this result, and discuss some interesting open questions that are suggested by our methods.

\section{Review of the case $G = \GL_n$}

We begin by describing what is known about the Langlands conjectures in the setting of the general linear group $G = \GL_n$. In this case very complete results were obtained by Laurent Lafforgue \cite{Laf02}. As in the introduction, we write $\bbF_q$ for the finite field with $q$ elements, and let $X$ be a smooth, projective, connected curve over $\bbF_q$.  The Langlands correspondence predicts a relation between representations of the absolute Galois group of $K$ and automorphic representations of the group $\GL_n(\bbA_K)$. We now describe each of these in turn.

Let $K^s$ be a fixed separable closure of $K$. We write $\Gamma_K = \Gal(K^s / K)$ for the absolute Galois group of $K$, relative to $K^s$. It is a profinite group. If $S \subset X$ is a finite set of closed points, then we write $K_S \subset K^s$ for the maximal extension of $K$ which is unramified outside $S$, and $\Gamma_{K, S} = \Gal(K_S / K)$ for its Galois group. This group has a geometric interpretation: if we set $U = X - S$, and write $\overline{\eta}$ for the geometric generic point of $U$ corresponding to $K^s$, then there is a canonical identification $\Gamma_{K, S} \cong \pi_1^\text{\'et}(U, \overline{\eta})$ of the Galois group with the \'etale fundamental group of the open curve $U$. 

Fix a prime $\ell \nmid q$ and a continuous character $\omega : \Gamma_K \to \overline{\bbQ}_\ell^\times$ of finite order. If $n \geq 1$ is an integer, then we write $\Gal_{n, \omega}$ for the set of conjugacy classes of continuous representations $\rho : \Gamma_K \to \GL_n(\overline{\bbQ}_\ell)$ with the following properties:
\begin{enumerate}
\item $\rho$ factors through a quotient $\Gamma_K \to \Gamma_{K, S}$, for some finite subscheme $S \subset X$.
\item $\det \rho = \omega$.
\item $\rho$ is irreducible.
\end{enumerate}

To describe automorphic representations, we need to introduce ad\`eles. If $v \in X$ is a closed point, then the local ring $\cO_{X, v}$ is a discrete valuation ring, and determines a valuation $\ord_v : K^\times \to \bbZ$ which we call a place of $K$. The completion $K_v$ of $K$ with respect to this valuation is a local field, which can be identified with the field of Laurent series $\bbF_{q_v}((t_v))$, where $\bbF_{q_v}$ is the residue field of $\cO_{X, v}$ and $t_v \in \cO_{X, v}$ is a uniformizing parameter. We write $\cO_{K_v} \subset K_v$ for the valuation ring.

The ad\`ele ring $\bbA_K$ is the restricted direct product of the rings $K_v$, with respect to their open compact subrings $\cO_{K_v}$. It is a locally compact topological ring. Taking ad\`ele points of $\GL_n$, we obtain the group $\GL_n(\bbA_K)$, which is a locally compact topological group, and which can itself be identified with the restricted direct product of the groups $\GL_n(K_v)$ with respect to their open compact subgroups $\GL_n(\cO_{K_v})$.

If $n = 1$, then $\GL_1(\bbA_K) = \bbA_K^\times$ and class field theory gives a continuous map
\[ \Art_K : K^\times \backslash \bbA_K^\times \to \Gamma_K^\text{ab}, \]
which is injective with dense image. We write $\cA_{n, \omega}$ for the $\overline{\bbQ}_\ell$-vector space of functions $f : \GL_n(\bbA_K) \to \overline{\bbQ}_\ell$ satisfying the following conditions:
\begin{enumerate}
\item $f$ is invariant under left translation by the discrete subgroup $\GL_n(K) \subset \GL_n(\bbA_K)$.
\item For any $z \in \bbA_K^\times$, $g \in \GL_n(\bbA_K)$, $f(gz) = \omega(\Art_K(z)) f(g)$.
\item $f$ is smooth, i.e.\ there exists an open compact subgroup $U \subset \GL_n(\bbA_K)$ such that for all $u \in U$, $g \in \GL_n(\bbA_K)$, $f(gu) = f(g)$.
\end{enumerate}
Then the group $\GL_n(\bbA_K)$ acts on $\cA_{n, \omega}$ by right translation. We write $\cA_{n, \omega, \text{cusp}} \subset \cA_{n, \omega}$ for the subspace of cuspidal functions, i.e.\ those satisfying the following additional condition:
\begin{enumerate}
\item[4.] For each proper parabolic subgroup $P \subset \GL_n$, of unipotent radical $N$, we have 
\[ \int_{n \in N(K) \backslash N(\bbA_K)} f(ng) \, dn = 0 \]
for all $g \in \GL_n(\bbA_K)$. (Note that the quotient $N(K) \backslash N(\bbA_K)$ is compact, so the integral, taken with respect to a quotient Haar measure on $N(\bbA_K)$, is well-defined.)
\end{enumerate}
With this definition, $\cA_{n, \omega, \text{cusp}} \subset \cA_{n, \omega}$ is an $\overline{\bbQ}_\ell[\GL_n(\bbA_K)]$-submodule. The following theorem describes the basic structure of this representation of $\GL_n(\bbA_K)$.
\begin{theorem}
\begin{enumerate} \item $\cA_{n, \omega, \text{cusp}}$ is a semisimple admissible $\overline{\bbQ}_\ell[\GL_n(\bbA_K)]$-module. Each irreducible constituent $\pi \subset  \cA_{n, \omega, \text{cusp}}$ appears with multiplicity 1. 
\item If $\pi \subset \cA_{n, \omega, \text{cusp}}$ is an irreducible submodule, then there is a  decomposition $\pi = \otimes'_v \pi_v$ of $\pi$ as a restricted tensor product of irreducible admissible representations $\pi_v$ of the groups $\GL_n(K_v)$ (where $v$ runs over the set of all places of $K$).
\end{enumerate}
\end{theorem}
If $\pi$ is an irreducible constituent of $\cA_{n, \omega, \text{cusp}}$, then we call $\pi$ a cuspidal automorphic representation of $\GL_n(\bbA_K)$. We write $\Aut_{n, \omega}$ for the set of isomorphism classes of cuspidal automorphic representations of $\GL_n(\bbA_K)$. 

We can now state Langlands reciprocity for $\GL_n$.
\begin{theorem}\label{thm_gl_n_preliminary}
There is a bijection $\Gal_{n, \omega} \leftrightarrow \Aut_{n, \omega}$.
\end{theorem}
In order for this theorem to have content, we need to describe how to characterize the bijection whose existence it asserts. The most basic characterization uses restriction to unramified places. Let $v$ be a place of $K$. If $\rho \in \Gal_{n, \omega}$, then we can consider its restriction $\rho_v = \rho|_{W_{K_v}}$ to the Weil group $W_{K_v} \subset \Gamma_{K_v}$.\footnote{Here $\Gamma_{K_v} = \Gal(K_v^s / K_v)$ is the absolute Galois group of $K_v$, with respect to a fixed choice of separable closure. An embedding $K^s \to K_v^s$ extending the map $K \to K_v$ determines an embedding $\Gamma_{K_v} \to \Gamma_K$. The Weil group $W_{K_v} \subset \Gamma_{K_v}$ is the subgroup of elements which act on the residue field by an integral power of the (geometric) Frobenius element $\Frob_v$; see for example \cite{Tat79} for a detailed discussion.} If $\pi \in  \Aut_{n, \omega}$, then we can consider the factor $\pi_v$, which is an irreducible admissible representation of the group $\GL_n(K_v)$. 
\begin{definition}
\begin{enumerate}
\item A continuous homomorphism $\rho_v : W_{K_v} \to \GL_n(\overline{\bbQ}_\ell)$ is said to be unramified if it factors through the unramified quotient $W_{K_v} \to \bbZ$.
\item An irreducible admissible representation of the group $\GL_n(K_v)$ is said to be unramified if the subspace $\pi_v^{\GL_n(\cO_{K_v})}$ of $\GL_n(\cO_{K_v})$-invariant vectors is non-zero. 
\end{enumerate}
\end{definition}
These two kinds of unramified objects are related by the \emph{unramified local Langlands correspondence}, which can be phrased as follows:
\begin{theorem}\label{thm_satake_isomorphism}
Let $v$ be a place of $K$. There is a canonical\footnote{As the proof shows, we need to fix as well a choice of a square root of $q$ in $\overline{\bbQ}_\ell$.} bijection $\pi_v \mapsto t(\pi_v)$ between the following two sets:
\begin{enumerate}
\item The set of isomorphism classes of unramified irreducible admissible representations $\pi_v$ of $\GL_n(K_v)$ over $\overline{\bbQ}_\ell$.
\item The set of semisimple conjugacy classes $t$ in $\GL_n(\overline{\bbQ}_\ell)$.
\end{enumerate}
\end{theorem}
\begin{proof}[Proof (sketch)]
The proof, which is valid for any reductive group $G$ over $\bbF_q$, goes via the Satake isomorphism. If $G$ is split then this is an isomorphism
\[ \cH(G(K_v), G(\cO_{K_v})) \otimes_\bbZ \bbZ[q_v^{\pm \frac{1}{2}}] \to \bbZ[ \hG ]^{\hG} \otimes \bbZ[q_v^{\pm \frac{1}{2}}], \]
\[ T_{v, f} \mapsto f \]
where $\cH$ is the Hecke algebra of $G(\cO_{K_v})$-biinvariant functions $f : G(K_v) \to \bbZ$ of compact support, and $\bbZ[\hG]^{\hG}$ is the algebra of conjugation-invariant functions on the dual group $\hG$ (which is $\GL_n$ if $G = \GL_n$). If $\pi_v$ is an irreducible admissible representation of $G(K_v)$ over $\overline{\bbQ}_l$ and $\pi_v^{G(\cO_{K_v})} \neq 0$, then $\pi_v^{G(\cO_{K_v})}$ is a simple $\cH \otimes_\bbZ \overline{\bbQ}_l$-module, which therefore determines a homomorphism $\bbZ[ \hG ]^{\hG} \to \overline{\bbQ}_\ell$. The geometric invariant theory of the adjoint quotient of the reductive group $\hG$ implies that giving such a homomorphism is equivalent to giving a conjugacy class of semisimple elements in $\hG(\overline{\bbQ}_\ell)$. 
\end{proof}
Let $S$ be a finite set of places of $K$. We write $\Gal_{n, \omega, S} \subset \Gal_{n, \omega}$ for the set of $\rho$ such that for each place $v \not\in S$ of $K$, $\rho|_{W_{K_v}}$ is unramified (we say that `$\rho$ is unramified outside $S$'). We write $\Aut_{n, \omega, S} \subset \Aut_{n, \omega}$ for the set of $\pi = \otimes'_v \pi_v$ such that for each place $v \not\in S$ of $K$, $\pi_v$ is unramified (we say that `$\pi$ is unramified outside $S$'). We can now state a more precise version of Theorem \ref{thm_gl_n_preliminary}:
\begin{theorem}\label{thm_weak_correspondence_for_GL_n}
Let $S$ be a finite set of places of $K$. Then there is a bijection $\pi \mapsto \rho(\pi) : \Aut_{n, \omega, S} \to \Gal_{n, \omega, S}$ with the following property: for each place $v \not\in S$, $\rho(\pi)(\Frob_v)^\text{ss} \in t(\pi_v)$\footnote{Here and elsewhere, we write $x^\text{ss}$ for the semisimple part in the Jordan decomposition $x = x^\text{ss} x^\text{u}$ of an element $x$ of a linear algebraic group.}. 
\end{theorem}
This defining property uniquely characterizes the bijection, if it exists. Indeed, the isomorphism class of any representation $\pi \in \Aut_{n, \omega, S}$ is uniquely determined by the representations $\pi_v$ ($v \not\in S$): this is the strong multiplicity one theorem. Similarly, any representation $\rho \in \Gal_{n, \omega, S}$ is uniquely determined by the conjugacy classes of the elements $\rho(\Frob_v)^\text{ss}$ ($v \not\in S$): the irreducible representation $\rho$ is uniquely determined up to isomorphism by its character $\tr \rho$. This continuous function $\tr \rho : \Gamma_{K, S} \to \overline{\bbQ}_\ell$ is determined by its values at a dense set of elements, and the Chebotarev density theorem implies that the Frobenius elements $\Frob_v$ $(v \not\in S$) form such a set. 

L. Lafforgue proved Theorem \ref{thm_weak_correspondence_for_GL_n} using an induction on $n$. If the theorem is known for $n' < n$, then the `principe de r\'ecurrence de Deligne' (see \cite[Appendice B]{Laf02}) reduces the problem to constructing, for any $\pi \in \Aut_{n, \omega, S}$, the corresponding Galois representation $\rho(\pi) \in \Gal_{n, \omega, S}$, as well as proving that certain $L$-- and $\epsilon$--factors are matched up under the correspondence. The three main ingredients that make this possible are Grothendieck's theory of $L$-functions of Galois representations, Laumon's product formula for the $\epsilon$--factors of Galois representations, and Piatetski-Shapiro's converse theorem, which can be used to show that an irreducible admissible representation of $\GL_n(\bbA_K)$ with sufficiently well-behaved associated $L$-functions is in fact cuspidal automorphic. We note that in carrying this out Lafforgue actually obtains a much more precise result than Theorem \ref{thm_weak_correspondence_for_GL_n}, in particular re-proving the local Langlands correspondence for $\GL_n$ and showing that the global correspondence is compatible with the local one. 

One would like to generalise Theorem \ref{thm_weak_correspondence_for_GL_n} to an arbitrary reductive group $G$ over $\bbF_q$. However, there are a number (!) of difficulties. To begin with, it is not even clear what the correct statement should be: it is easy to write down the naive analogues of the sets $\Aut_{n, \omega}$ and $\Gal_{n, \omega}$, but we will see some reasons why they cannot be related by a simple bijection.  Moreover, no converse theorem is known for a general group $G$, which means there is no apparent way of proving that a given admissible representation of $G(\bbA_K)$ is in fact cuspidal automorphic. This is the motivation behind proving a result like our main theorem. 

\section{Pseudocharacters}

Let us now pass to the case of a general reductive group $G$ over $\bbF_q$. In order to simplify the discussion here, we will assume that $G$ is split. In this case one can associate to $G$ its dual group $\widehat{G}$, a split reductive group over $\bbZ$, which is characterized by the property that its root datum is dual to that of $G$ (see e.g. \cite{Bor79}). If $G = \GL_n$, then $\widehat{G} = \GL_n$, so our discussion will include $n$-dimensional linear representations as a special case. 

We will describe the results of Vincent Lafforgue's construction in the next section. First, we make a detour to describe the notion of $\hG$-pseudocharacter, which was introduced for the first time in \cite{Laf12}. This is a generalization of the notion of the pseudocharacter of an $n$-dimensional representation (to which it reduces in the case $G = \GL_n$).

Let $\Gamma$ be a group, and let $\Omega$ be an algebraically closed field of characteristic 0. We recall that to any representation $\rho : \Gamma \to \GL_n(\Omega)$, we can associate the character $\tr \rho : \Gamma \to \Omega$; it clearly depends only on $\rho$ up to conjugacy and up to semisimplification. We have the following theorem, the second part of which was proved by Taylor using results of Procesi \cite{Tay91, Pro76}.
\begin{theorem}\label{thm_pseudocharacters_for_GL_n}
\begin{enumerate}
\item Let $\rho, \rho' : \Gamma \to \GL_n(\Omega)$ be semisimple representations. Then they are isomorphic if and only if $\tr \rho = \tr \rho'$.
\item Let $t : \Gamma \to \Omega$ be a function satisfying the following conditions:
\begin{enumerate}
\item $t(1) = n$.
\item For all $\gamma_1, \gamma_2 \in \Gamma$, $t(\gamma_1 \gamma_2) = t(\gamma_2 \gamma_1)$.
\item For all $\gamma_1, \dots, \gamma_{n+1} \in \Gamma$, $\sum_{\sigma \in S_{n+1}} t_\sigma(\gamma_1, \dots, \gamma_{n+1}) = 0$, where if $\sigma$ has cycle decomposition 
\[ \sigma = (a_1 \dots a_{k_1}) (b_1 \dots b_{k_2}) \dots\]
 then we set 
 \[ t_\sigma(\gamma_1, \dots, \gamma_{n+1}) = t( \gamma_{a_1} \dots \gamma_{a_{k_1}} ) t( \gamma_{b_1} \dots \gamma_{b_{k_2}}) \dots. \]
\end{enumerate}
Then there exists a representation $\rho : \Gamma \to \GL_n(\Omega)$ such that $\tr \rho = t$.
\end{enumerate}
\end{theorem}
We can call a function $t : \Gamma \to \Omega$ satisfying the condition of Theorem \ref{thm_pseudocharacters_for_GL_n} a pseudocharacter of dimension $n$. Then the theorem says that sets of conjugacy classes of semisimple representations $\rho : \Gamma \to \GL_n(\Omega)$ and of pseudocharacters of dimension $n$ are in canonical bijection.

Here is Lafforgue's definition of a $\hG$-pseudocharacter. Let $A$ be a ring.
\begin{definition}
Let $\mathbf{t} = (t_n)_{n \geq 1}$ be a collection of algebra maps $t_n : \bbZ[\hG^n]^\hG \to \operatorname{Fun}(\Gamma^n, A)$ satisfying the following conditions: 
\begin{enumerate}
\item For each $n, m \geq 1$ and for each $\zeta : \{ 1, \dots, m \} \to \{1, \dots, n \}$,  $f \in \bbZ[\hG^m]^\hG$, and $\gamma = (\gamma_1, \dots, \gamma_n) \in \Gamma^n$, we have
\[ t_n(f^\zeta)(\gamma_1, \dots, \gamma_n) = t_m(f)(\gamma_{\zeta(1)}, \dots, \gamma_{\zeta(m)}), \]
where $f^\zeta(g_1, \dots, g_n) = f(g_{\zeta(1)}, \dots, g_{\zeta(m)})$.
\item For each $n \geq 1$, $\gamma = (\gamma_1, \dots, \gamma_{n+1}) \in \Gamma^{n+1}$, and $f \in \bbZ[\hG^n]^\hG$, we have
\[ t_{n+1}(\hat{f})(\gamma_1, \dots, \gamma_{n+1}) = t_n(f)(\gamma_1, \dots, \gamma_{n-1}, \gamma_n \gamma_{n+1}), \]
where $\hat{f}(g_1, \dots, g_{n+1}) = f(g_1, \dots, g_{n-1}, g_n g_{n+1})$.
\end{enumerate}
Then $\mathbf{t}$ is called a $\hG$-pseudocharacter of $\Gamma$ over $A$.
\end{definition}
Note that $\hG$ acts on $\hG^n$ by diagonal conjugation. The subring $\bbZ[\hG^n]^\hG \subset \bbZ[\hG^n]$ is the ring of functions invariant under this action. We observe that if $\rho : \Gamma \to \hG(A)$ is a homomorphism, then we can define a $\hG$-pseudocharacter $\tr \rho = (t_n)_{n \geq 1}$ of $\Gamma$ over $A$ by the formula 
\[ t_n(f)(\gamma_1, \dots, \gamma_n) = f(\rho(\gamma_1), \dots, \rho(\gamma_n)). \]
It is clear that this depends only on the $\hG(A)$-conjugacy class of $\rho$.
\begin{theorem}\label{thm_pseudocharacters_biject_with_representations}
Let $\Gamma$ be a group, and let $\Omega$ be an algebraically closed field.
\begin{enumerate}
\item Let $\rho, \rho' : \Gamma \to \widehat{G}(\Omega)$ be $\hG$-completely reducible representations.\footnote{A representation $\rho$ is said to be $\hG$-irreducible if its image is contained in no proper parabolic subgroup of $\hG_\Omega$, and $\hG$-completely reducible if for any parabolic subgroup $P \subset \hG_\Omega$ containing the image, there exists a Levi subgroup $L \subset P$ such that $\rho(\Gamma) \subset L(\Omega)$. See e.g. \cite{Ser05}.} Then $\rho, \rho'$ are $\widehat{G}(\Omega)$-conjugate if and only if $\tr \rho = \tr \rho'$.
\item Let $\mathbf{t}$ be a $\widehat{G}$-pseudocharacter. Then there exists a representation $\rho : \Gamma \to \widehat{G}(\Omega)$ such that $\mathbf{t} = \tr \rho$.
\end{enumerate}
\end{theorem}
The proof of Theorem \ref{thm_pseudocharacters_biject_with_representations} is based on Richardson's results about the geometric invariant theory of the action of $\hG$ on $\hG^n$ by diagonal conjugation \cite{Ric88}.

In the case where $\Gamma$ is profinite, we want to impose continuity conditions on its $\hG$-pseudocharacters. Fortunately, $\hG$-pseudocharacters are well-behaved from this point of view.
\begin{definition}
Let $A$ be a topological ring, and let $\Gamma$ be a profinite group. We say that a $\hG$-pseudocharacter $\mathbf{t}$ of $\Gamma$ over $A$ is continuous if each map $t_n : \bbZ[\hG^n]^\hG \to \operatorname{Fun}(\Gamma^n, A)$ takes values in the subset of continuous functions $\operatorname{Fun}_{\operatorname{cts}}(\Gamma^n, A)$. 
\end{definition}
\begin{proposition}
Let $\ell$ be a prime, and let $\Gamma$ be a profinite group. Let $\Omega = \overline{\bbQ}_\ell$ (with its $\ell$-adic topology) or $\overline{\bbF}_\ell$ (with the discrete topology). Let $\rho : \Gamma \to \hG(\Omega)$ be a $\hG$-completely reducible representation. Then $\rho$ is continuous if and only if $\tr \rho$ is continuous.
\end{proposition}
Finally, we note that $\hG$-pseudocharacters are well-behaved from the point of view of reduction modulo $\ell$. We will need this in our discussion of the deformation theory of pseudocharacters later on. 
\begin{proposition}
Let $\ell$ be a prime, and let $\Gamma$ be a profinite group. Then:
\begin{enumerate}
\item Let $\mathbf{t}$ be a continuous $\hG$-pseudocharacter of $\Gamma$ over $\overline{\bbQ}_\ell$. Then $\mathbf{t}$ takes values in $\overline{\bbZ}_\ell$ and  $\overline{\mathbf{t}}$, its reduction modulo $\ell$, is a continuous $\hG$-pseudocharacter over $\overline{\bbF}_\ell$.
\item Let $\rho : \Gamma \to \hG(\overline{\bbQ}_\ell)$ be a continuous representation. After replacing $\rho$ by a $\hG(\overline{\bbQ}_\ell)$-conjugate, we can assume that $\rho$ takes values in $\hG(\overline{\bbZ}_\ell)$. Let $\overline{\rho} : \Gamma \to \hG(\overline{\bbF}_\ell)$ denote the semisimplification of the reduction of $\rho$ modulo $\ell$. Then $\overline{\rho}$ depends only on $\rho$ up to $\hG(\overline{\bbQ}_\ell)$-conjugacy, and $\tr \overline{\rho} = \overline{\tr \rho}$.
\end{enumerate}
\end{proposition}
We now come back to our original case of interest, namely pseudocharacters of the group $\Gamma = \Gamma_{K, S}$. We can define a compatible family of pseudocharacters of $\Gamma_{K, S}$ of dimension $n$ to consist of the data of a number field $E$ and, for each prime-to-$q$ place $\lambda$ of $E$, a pseudocharacter $t_\lambda : \Gamma_{K, S} \to E_\lambda$ of dimension $n$. These should satisfy the following property:
\begin{itemize}
\item For each place $v \not\in S$ of $K$, the number $t_\lambda(\Frob_v)$ lies in $E \subset E_\lambda$ and is independent of the choice of $\lambda$.
\end{itemize}
Thus, for example, \cite[Conjecture I.2.10]{Del80} asks that every pseudocharacter $t_\ell : \Gamma_{K, S} \to \overline{\bbQ}_\ell$ satisfying certain conditions should be a member of a compatible family. This leads us to our first question: 
\begin{question}\label{q_compatible_family_of_pseudocharacters}
Is it possible to define a notion of `compatible family of $\hG$-pseudocharacters', generalizing the above notion for $\GL_n$?
\end{question}
If we are willing to consider instead compatible families of representations, then Drinfeld \cite{Dri17} gives a satisfying (positive) answer to Question \ref{q_compatible_family_of_pseudocharacters} using the results of \cite{Laf02}. The question remains, however, of whether we can phrase this for pseudocharacters in elementary terms and, in particular, whether it is possible to make sense of compatible families when $K$ is instead a global field of characteristic 0 (i.e.\ a number field).

Here is one case it is easy to make sense of the notion of compatible family:
\begin{proposition}\label{prop_Zariski_dense_image}
Let $\ell \nmid q$ be a prime, and let $\rho : \Gamma_{K, S} \to \hG(\overline{\bbQ}_\ell)$ be a continuous representation of Zariski dense image. Assume that $\hG$ is semisimple. Then we can find a number field $E$ and an embedding $E \hookrightarrow \overline{\bbQ}_\ell$, inducing the place $\lambda_0$ of $E$, with the following properties:
\begin{itemize}
\item For each place $v \not\in S$ of $K$ and for each $f \in \bbZ[\hG]^\hG$, $f(\rho(\Frob_v)) \in E$. In other words, the conjugacy class of $\rho(\Frob_v)^\text{ss}$ is defined over $E$. 
\item For each prime-to-$q$ place $\lambda$ of $E$, there exists a continuous homomorphism $\rho_\lambda : \Gamma_{K, S} \to \hG(E_\lambda)$ of Zariski dense image such that for each place $v\not\in S$ of $K$ and for each $f \in \bbZ[\hG]^\hG$, $f(\rho_\lambda(\Frob_v))$ lies in $E$ and equals $f(\rho(\Frob_v))$. In other words, $\rho_\lambda(\Frob_v)^\text{ss}$ lies in the same geometric conjugacy class as $\rho(\Frob_v)^\text{ss}$.
\end{itemize}
Furthermore, for any prime-to-$q$ place $\lambda$ of $E$ and for any continuous homomorphism $\rho'_\lambda :  \Gamma_{K, S} \to \hG(\overline{E}_\lambda)$ such that for each place $v\not\in S$ of $K$ and for each $f \in \bbZ[\hG]^\hG$, $f(\rho'_\lambda(\Frob_v))$ lies in $E$ and equals  $f(\rho(\Frob_v))$, $\rho'_\lambda$ is $\hG(\overline{E}_\lambda)$-conjugate to $\rho_\lambda$. In particular, $\rho$ and $\rho_{\lambda_0}$ are $\hG(\overline{\bbQ}_\ell)$-conjugate.
\end{proposition}
The proof makes use of the proof of the global Langlands correspondence for $\GL_n$ by L. Lafforgue \cite{Laf02}, together with Chin's application of this work to the analysis of compatible families \cite{Chi04}; see \cite[\S 6]{Boc17}.
\section{The work of V. Lafforgue}
Having introduced the notion of $\hG$-pseudocharacter, we can now describe the basic shape of Vincent Lafforgue's results in \cite{Laf12}. We recall that $G$ is a split reductive group over $\bbF_q$. In order to simplify statements, we are now going to impose the further assumption that $G$ has finite centre (i.e.\ is semisimple).\footnote{By contrast, the paper \cite{Laf12} does not impose any restriction on $G$; see in particular \S 12 of \emph{op. cit.}} We write $\cA_G$ for the $\overline{\bbQ}_\ell$-vector space of functions $f : G(\bbA_K) \to \overline{\bbQ}_\ell$ satisfying the following conditions:
\begin{enumerate}
\item $f$ is invariant under left translation by the discrete subgroup $G(K) \subset G(\bbA_K)$.
\item $f$ is smooth. 
\end{enumerate}
Then the group $G(\bbA_K)$ acts on $\cA_G$ by right translation. We write $\cA_{G, \text{cusp}} \subset \cA_G$ for the subspace of cuspidal functions, i.e.\ those satisfying the following condition:
\begin{enumerate}
\item[3.] For any proper parabolic subgroup $P \subset G$ of unipotent radical $N$, we have
\[ \int_{n \in N(K) \backslash N(\bbA_K)} f(ng) \, dn = 0 \]
for all $g \in G(\bbA_K)$.
\end{enumerate}
With this definition, $\cA_{G, \text{cusp}}$ is a semisimple admissible $\overline{\bbQ}_\ell[G(\bbA_K)]$-module. In general, understanding the constituents of this space is much more complicated than for the group $\GL_n$. For example:
\begin{itemize}
\item Multiplicity one does not hold: there can exist representations $\pi$ of $G(\bbA_K)$ which appear in $\cA_{G, \text{cusp}}$ with multiplicity greater than 1.
\item Strong multiplicity one does not hold: there can exist representation $\pi, \pi'$ of $G(\bbA_K)$ which have positive multiplicity in $\cA_{G, \text{cusp}}$, such that $\pi_v \cong \pi_v'$ for all but finitely many places $v$ of $K$, but such that $\pi \not\cong \pi'$.
\end{itemize}
These phenomena are reflected in what happens on the Galois side. For example:
\begin{itemize}
\item There can exist everywhere unramified homomorphisms $\rho, \rho' : \Gamma_K \to \widehat{G}(\overline{\bbQ}_\ell)$ such that $\rho(\Frob_v)$ and $\rho'(\Frob_v)$ are conjugate for every $v$, but such that $\rho, \rho'$ are not conjugate.
\item There can exist homomorphisms $\rho, \rho' : \Gamma_K \to \widehat{G}$ and a place $v_0$ of $K$ such that for all $v \neq v_0$, $\rho$ and $\rho'$ are unramified at $v$ and $\rho(\Frob_v)$, $\rho'(\Frob_v)$ are $\hG(\overline{\bbQ}_\ell)$-conjugate; but $\rho|_{\Gamma_{K_{v_0}}} \not\cong \rho'|_{\Gamma_{K_{v_0}}}$. (Another reason for the failure of strong multiplicity one, not related to this Galois--theoretic phenomenon, is the existence of non-trivial $L$-packets.)
\end{itemize}
We refer the reader to \cite{Wan12} for a survey of how the relation between these phenomena can be understood in terms of Arthur's conjectural decomposition of the space of automorphic forms in terms of $A$-parameters \cite{Art89a}. 

Lafforgue's construction, quite remarkably, gives a decomposition of the space $\cA_{G, \text{cusp}}$ of cusp forms on $G$ which is quite close in appearance to that predicted by Arthur. He defines for each $n \geq 1$, function $f \in \bbZ[\hG^n]^{\hG}$, and tuple of elements $\gamma = (\gamma_1, \dots, \gamma_n) \in \Gamma_K^n$, an operator $S_{n, f, \gamma} \in \End_{\overline{\bbQ}_\ell}(\cA_{G, \text{cusp}})$. He calls these `excursion operators', and proves the following two theorems: 
\begin{theorem}
\begin{enumerate}
\item The operators $S_{n, f, \gamma}$ commute with each other and with the action of $G(\bbA_K)$. 
\item Let $\cB \subset \End_{\overline{\bbQ}_\ell}(\cA_{G, \text{cusp}})$ denote the $\overline{\bbQ}_\ell$-subalgebra generated by the operators $S_{n, f, \gamma}$ for all possible choices of $n$, $f$, and $\gamma$. Then the system of maps $\mathbf{t} = (t_n)_{n \geq 1}$ given by
\[ t_n : \bbZ[\hG^n]^\hG \to \operatorname{Fun}(\Gamma_K^n, \cB), \]
\[ f \mapsto (\gamma \mapsto S_{n, f, \gamma}) \]
is a $\hG$-pseudocharacter of $\Gamma_K$ valued in $\cB$. 
\end{enumerate}
\end{theorem}
\begin{theorem}
Let $S$ be a finite set of places of $K$, and let $U \subset G(\bbA_K)$ be an open compact subgroup such that for each place $v\not\in S$ of $K$, $U_v = G(\cO_{K_v})$. Let $\cB_U$ denote the quotient of $\cB$ which acts faithfully on $\cA_{G, \text{cusp}, U}$. Then:
\begin{enumerate}
\item The pushforward of $\mathbf{t}$ along $\cB \to \cB_U$ is pulled back from a $\hG$-pseudocharacter $\mathbf{t}_U$ of $\Gamma_{K, S}$ valued in $\cB_U$.
\item If $v\not\in S$, then the image of $S_{1, f, \Frob_v}$ in $\cB_U$ equals the unramified Hecke operator $T_{v, f}$ (defined as in the proof of Theorem \ref{thm_satake_isomorphism}, via the Satake isomorphism).
\end{enumerate}
\end{theorem}
Since the algebra $\cB_U$ contains the unramified Hecke operators, it can be viewed as an enlargement of the usual Hecke algebra. 
\begin{corollary}
Let $\pi$ be a cuspidal automorphic representation of $G(\bbA_K)$, and let $V_\pi \subset \cA_{G, \text{cusp}}$ be the $\pi$-isotypic component. Let $\cB_\pi$ denote the quotient of $\cB$ which acts faithfully on $V_\pi$. Then $\cB_\pi$ is a finite-dimensional $\overline{\bbQ}_\ell$-algebra and for each maximal ideal $\frp$, one can associate a continuous representation $\sigma_{\pi, \frp} : \Gamma_K \to \widehat{G}(\overline{\bbQ}_\ell)$ with the following properties:
\begin{enumerate}
\item $\tr \sigma_{\pi, \frp} = \mathbf{t}_\pi \text{ mod }\frp$.
\item Let $S$ be a finite set of places of $K$ such that $\pi_v^{G(\cO_{K_v})} \neq 0$ if $v \not\in S$. Then $\sigma_{\pi, \frp}$ is unramified outside $S$ and if $v \not\in S$, then $\sigma_{\pi, \frp}(\Frob_v)^\text{ss} \in t(\pi_v)$.
\item If $\frp \neq \frp'$, then $\sigma_{\pi, \frp} \not\cong \sigma_{\pi, \frp'}$.
\end{enumerate}
\end{corollary}
\begin{proof}
Since $V_\pi$ has finite length as a $\overline{\bbQ}_\ell[G(\bbA_K)]$-module and $\cB_\pi $ is contained inside $\End_{\overline{\bbQ}_\ell[G(\bbA_K)]}(V_\pi)$, $\cB_\pi$ is a finite-dimensional $\overline{\bbQ}_\ell$-algebra. Since $\cB_\pi$ is a quotient of $\cB$, it carries a $\hG$-pseudocharacter $\mathbf{t}_\pi$. Each maximal ideal $\frp \subset \cB_\pi$ has residue field $\overline{\bbQ}_\ell$, and the pushforward of $\mathbf{t}_\pi$ along the map $\cB_\pi \to \cB_\pi / \frp \cong \overline{\bbQ}_\ell$ therefore corresponds, by Theorem \ref{thm_pseudocharacters_biject_with_representations}, to a continuous $\hG$-completely reducible representation $\sigma_{\pi, \frp} : \Gamma_{K, S} \to \hG(\overline{\bbQ}_\ell)$ satisfying the following property: for all $n \geq 1$, $f \in \bbZ[\hG^n]^\hG$, $\gamma = (\gamma_1, \dots, \gamma_n) \in \Gamma_{K, S}^n$, we have
\[ f(\sigma_{\pi, \frp}(\gamma_1), \dots, \sigma_{\pi, \frp}(\gamma_n)) = S_{n, f, \gamma} \text{ mod }\frp. \]
From this identity it is apparent that $\sigma_{\pi, \frp}$ determines $\frp$. Specializing to $n = 1$ and $\gamma = \Frob_v$ for some $v \not\in S$, this identity reduces to the formula
\[ f(\sigma_{\pi, \frp}(\Frob_v)) = T_{v, f} \text{ mod }\frp, \]
or equivalently that $\sigma_{\pi, \frp}(\Frob_v)^\text{ss}$ is in the conjugacy class $t(\pi_v)$.
\end{proof}
\begin{question}\label{q_excursion_algebra_independent_of_l}
The space $V_\pi$ can be defined over $\overline{\bbQ}$. Is there a sense in which its decomposition $V_\pi = \oplus V_{\pi, \frp}$ is independent of $\ell$? 
\end{question}
Presumably a positive answer to this question must be tied up with a positive answer to Question \ref{q_compatible_family_of_pseudocharacters}. 

We can now define what it means for a Galois representation to be cuspidal automorphic, in the sense of the algebra $\cB$.
\begin{definition}\label{def_cuspidal_automorphic_galois_representation}
We say that a representation $\rho : \Gamma_K \to \widehat{G}(\overline{\bbQ}_\ell)$ is cuspidal automorphic if there exists a cuspidal automorphic representation $\pi$ of $G(\bbA_K)$ and a maximal ideal $\frp \subset \cB_\pi$ such that $\rho \cong \sigma_{\pi, \frp}$. 
\end{definition}
Note that this definition depends in an essential way on Lafforgue's excursion operators! 

We are now in a position to state the main theorem of \cite{Boc17}:
\begin{theorem}\label{thm_BHKT}
Let $\rho : \Gamma_{K, \emptyset} \to \widehat{G}(\overline{\bbQ}_\ell)$ be a continuous representation of Zariski dense image. Then $\rho$ is \emph{potentially cuspidal automorphic}: there exists a finite Galois extension $L / K$ such that $\rho|_{\Gamma_{L, \emptyset}}$ is cuspidal automorphic in the sense of Definition \ref{def_cuspidal_automorphic_galois_representation}. 
\end{theorem}
\begin{corollary}
Let $\rho : \Gamma_{K, \emptyset} \to \widehat{G}(\overline{\bbQ}_\ell)$ be a continuous representation of Zariski dense image. Then there exists a finite Galois extension $L / K$ and an everywhere unramified cuspidal automorphic representation $\pi$ of $G(\bbA_L)$ such that for each place $w$ of $L$, $\rho|_{W_{L_w}}$ and $\pi_w$ are related under the unramified local Langlands correspondence for $G(L_w)$.
\end{corollary}
(In fact, Proposition \ref{prop_Zariski_dense_image} implies that the theorem and its corollary are equivalent.) In the remainder of this article we will sketch the proof of Theorem \ref{thm_BHKT}.
\section{An automorphy lifting theorem for $G$}

How can one show that a Galois representation $\rho : \Gamma_{K, \emptyset} \to \widehat{G}(\overline{\bbQ}_\ell)$ is automorphic, in the sense of Definition \ref{def_cuspidal_automorphic_galois_representation}? For a general group $G$, we no longer know how to construct automorphic forms using converse theorems.

We pursue a different path which is inspired by the proofs of existing potential automorphy results for Galois representations $\Gamma_E \to \GL_n$, where $E$ is a number field. These are in turn based on automorphy lifting theorems, which are provable instances of the following general principle:
\begin{principle}\label{principle_ALT}
Let $\rho, \rho' : \Gamma_{K, \emptyset} \to \widehat{G}(\overline{\bbQ}_\ell)$ be continuous representations and let $\overline{\rho}$, $\overline{\rho}' : \Gamma_{K, \emptyset} \to \hG(\overline{\bbF}_\ell)$ denote their reductions modulo $\ell$. Suppose that $\overline{\rho}$, $\overline{\rho}'$ are $\hG(\overline{\bbF}_\ell)$-conjugate and $\hG$-irreducible. Suppose that $\rho$ is cuspidal automorphic. Then $\rho'$ is also cuspidal automorphic. 
\end{principle}
The first theorem of this type was stated by Wiles on his way to proving Fermat's Last Theorem \cite{Wil95}. Our proof of an analogous result is inspired by Diamond's elaboration of the Taylor--Wiles method \cite{Dia97}, which gives a way to construct an isomorphism $R \cong \bbT$, where $R$ is a Galois deformation ring and $\bbT$ is a Hecke algebra acting on cuspidal automorphic forms. By contrast, we prove an `$R = \cB$' theorem, where $\cB$ is a suitable ring of Lafforgue's excursion operators. 

We describe these objects in order to be able to state a precise result. We will stick to the everywhere unramified case. We first consider the Galois side. Let $k \subset \overline{\bbF}_l$ be a finite subfield, and let $\overline{\rho} : \Gamma_{K, \emptyset} \to \widehat{G}(k)$ be a continuous homomorphism. Let $\Art_k$ denote the category of Artinian local $W(k)$-algebras $A$, equipped with an isomorphism $A / \ffrm_A \cong k$. We define $\operatorname{Lift}_{\overline{\rho}} : \Art_k \to \Sets$ to be the functor of liftings of $\overline{\rho}$, i.e.\ of homomorphisms $\rho_A : \Gamma_{K, \emptyset} \to \hG(A)$ such that $\rho_A \text{ mod }\ffrm_A = \overline{\rho}$. 

For any $A \in \Art_k$, the group $\ker(\widehat{G}(A) \to \widehat{G}(k))$ acts on $\Lift_{\overline{\rho}}(A)$ by conjugation, and we write $\Def_{\overline{\rho}} : \Art_k \to \Sets$ for the quotient functor (given by the formula $\Def_{\overline{\rho}}(A) = \Lift_{\overline{\rho}}(A) / \ker(\widehat{G}(A) \to \widehat{G}(k))$). The following lemma is basic.
\begin{lemma}
Suppose that $\overline{\rho}$ is absolutely $\hG$-irreducible, and that $\ell$ does not divide the order of the Weyl group of $\hG$. Then the functor $\Def_{\overline{\rho}}$ is pro-represented by a complete Noetherian local $W(k)$-algebra $R_{\overline{\rho}}$ with residue field $k$.
\end{lemma}
In order to be able to relate the deformation ring $R_{\overline{\rho}}$ to automorphic forms, we need to introduce integral structures.  We therefore write $\cC_{G, k}$ for the set of functions $f : G(\bbA_K) \to W(k)$ satisfying the following conditions:
\begin{enumerate}
\item $f$ is invariant under left translation by $G(K)$.
\item $f$ is smooth.
\end{enumerate}
We write $\cC_{G, k, \text{cusp}}$ for the intersection $\cC_{G, k} \cap \cA_{G, \text{cusp}}$ (taken inside $\cA_{G}$).  Let $U = \prod_v G(\cO_{K_v})$.
\begin{proposition}
Suppose that $f \in \bbZ[\hG^n]^{\hG}$. Then each operator $S_{n, f, \gamma} \in \cB_U \subset \End_{\overline{\bbQ}_\ell}(\cA_{G, \text{cusp}}^U)$ leaves invariant the submodule $\cC_{G, k, \text{cusp}}^U$.
\end{proposition}
We write $\cB(U, W(k))$ for the $W(k)$-subalgebra of $\End_{W(k)}(\cC_{G, k, \text{cusp}}^U)$ generated by the operators $S_{n, f, \gamma}$ for $f \in \bbZ[\hG^n]^{\hG}$. Then $\cB(U, W(k))$ is a finite flat $W(k)$-algebra and there is a $\hG$-pseudocharacter $\mathbf{t}_{U, W(k)}$ of $\Gamma_{K, \emptyset}$ valued in $\cB(U, W(k))$.

Let $\ffrm$ be a maximal ideal of $\cB(U, W(k))$. Its residue field is a finite extension of the finite field $k$. After possibly enlarging $k$, we can assume that the following conditions hold:
\begin{itemize}
\item The residue field of $\ffrm$ equals $k$.
\item There exists a continuous representation $\overline{\rho}_\ffrm : \Gamma_{K, \emptyset} \to \hG(k)$ such that $\tr \overline{\rho}_\ffrm = \mathbf{t}_{U, W(k)} \text{ mod }\ffrm$.
\end{itemize}
Then the ring $\cB(U, W(k))_\ffrm$ (localization at the maximal ideal $\ffrm$) is a finite flat local $W(k)$-algebra of residue field $k$, and it comes equipped with a pseudocharacter $\mathbf{t}_{U, W(k), \ffrm}$. A natural question to ask is: under what conditions does this pseudocharacter arise from a representation $\rho_\ffrm : \Gamma_{K, \emptyset} \to \widehat{G}(\cB(U, W(k))_\ffrm)$ lifting $\overline{\rho}_\ffrm$? In other words, under what conditions does the analogue of Theorem \ref{thm_pseudocharacters_biject_with_representations} hold when we no longer restrict to field-valued $\hG$-pseudocharacters? 
\begin{proposition}\label{prop_characters_and_pseudocharacters}
Suppose that $\overline{\rho}_\ffrm$ is absolutely $\hG$-irreducible, and that its centralizer $\Cent(\hG^\text{ad}_k, \overline{\rho}_\ffrm)$ is scheme-theoretically trivial.\footnote{Here and elsewhere, $\hG^\text{ad}_k$ denotes the adjoint group of $\hG_k$, i.e.\ the quotient of $\hG_k$ by its centre.} Suppose that $\ell$ does not divide the order of the Weyl group of $\hG$. Then there is a unique conjugacy class of liftings $[\rho_\ffrm] \in \Def_{\overline{\rho}_\ffrm}(\cB(U, W(k))_\ffrm)$ such that $\tr \rho_\ffrm = \mathbf{t}_{U, W(k), \ffrm}$.
\end{proposition}
Under the assumptions of Proposition \ref{prop_characters_and_pseudocharacters}, we see the ring $R_{\overline{\rho}_\ffrm}$ is defined, and that its universal property determines a canonical map $R_{\overline{\rho}_\ffrm} \to \cB(U, W(k))_\ffrm$. The localized space $(\cC_{G, k, \text{cusp}}^U)_\ffrm$ of automorphic forms then becomes a module for the deformation ring $R_{\overline{\rho}_\ffrm}$.

We are now in a position to state a provable instance of Principle \ref{principle_ALT}.
\begin{theorem}\label{thm_automorphy_lifting}
Let $\ffrm \subset \cB(U, W(k))$ be a maximal ideal of residue field $k$, and suppose that there exists a continuous, absolutely $\hG$-irreducible representation $\overline{\rho}_{\ffrm} : \Gamma_{K, \emptyset} \to \hG(k)$ such that $\tr \overline{\rho}_{\ffrm} = \mathbf{t}_{U, W(k)} \text{ mod }\ffrm$. Suppose further that the following conditions are satisfied:
\begin{enumerate}
\item $\ell > \# W$, where $W$ is the Weyl group of the split reductive group $\hG$.
\item The centralizer $\Cent(\hG^\text{ad}_k, \overline{\rho}_\ffrm)$ is scheme-theoretically trivial.
\item The representation $\overline{\rho}_\ffrm$ is absolutely strongly $\hG$-irreducible.
\item The subgroup $\overline{\rho}_\ffrm(\Gamma_{K(\zeta_\ell)}) \subset \hG(k)$ is $\hG$-abundant.
\end{enumerate}
Then $(\cC_{G, k, \text{cusp}}^U)_\ffrm$ is a finite free $R_{\overline{\rho}_\ffrm}$-module.
\end{theorem}
\begin{corollary}
With the assumptions of the theorem, let $\rho : \Gamma_{K, \emptyset} \to \hG(\overline{\bbQ}_\ell)$ be a continuous homomorphism such that $\overline{\rho} \cong \overline{\rho}_\ffrm$. Then $\rho$ is cuspidal automorphic. 
\end{corollary}
\begin{proof}
The theorem implies that the map $R_{\overline{\rho}_\ffrm} \to \cB(U, W(k))_\ffrm$ that we have constructed is an isomorphism. Any representation $\rho$ as in the statement of the corollary determines a homomorphism $R_{\overline{\rho}_\ffrm} \to \overline{\bbQ}_\ell$. (To show this, we first need to prove that a conjugate of $\rho$ takes values in $\hG(\cO)$, where $\cO$ is a complete Noetherian local $W(k)$-subalgebra of $\overline{\bbQ}_\ell$ of residue field $k$.) This in turn determines a homomorphism $\cB(U, W(k))_\ffrm \to \overline{\bbQ}_\ell$, hence a maximal ideal $\frp \subset \cB_U$ with the property that for each $n \geq 1$, $f \in \bbZ[\hG^n]^\hG$ and $\gamma = (\gamma_1, \dots, \gamma_n) \in \Gamma_{K, \emptyset}^n$, 
\[ f(\rho(\gamma_1), \dots, \rho(\gamma_n)) = S_{n,f,\gamma} \text{ mod }\frp. \]
This is exactly what it means for $\rho$ be cuspidal automorphic.
\end{proof}
There are two adjectives in the theorem that have yet to be defined: `strongly $\hG$-irreducible' and `$\hG$-abundant'. We remedy this now:
\begin{definition}
Let $\Omega$ be an algebraically closed field, and let $\Gamma$ be a group. We say that a homomorphism $\sigma : \Gamma \to \hG(\Omega)$ is strongly $\hG$-irreducible if for any other homomorphism $\sigma' : \Gamma \to \hG(\Omega)$ such that for all $\gamma \in \Gamma$, $\sigma(\gamma)^\text{ss}$ and $\sigma'(\gamma)^\text{ss}$ are $\hG(\Omega)$-conjugate, $\sigma'$ is $\hG$-irreducible.
\end{definition}
Thus a strongly $\hG$-irreducible representation is $\hG$-irreducible. We do not know an example of a representation which is $\hG$-irreducible but not strongly $\hG$-irreducible. 
\begin{definition}
Let $k$ be a finite field, and let $H \subset \hG(k)$ be a subgroup. We say that $H$ is $\hG$-abundant if the following conditions are satisfied:
\begin{enumerate}
\item The cohomology groups $H^0(H, \widehat{\frg}_k)$, $H^0(H, \widehat{\frg}_k^\vee)$, $H^1(H, \widehat{\frg}_k^\vee)$ and $H^1(H, k)$ all vanish. (Here $\widehat{\frg}_k$ denotes the Lie algebra of $\hG_k$, and $\widehat{\frg}_k^\vee$ its dual.)
\item For each regular semisimple element $h \in H$, the torus $\Cent(\hG_k, h)^\circ$ is split.
\item For each simple $k[H]$-submodule $W \subset \widehat{\frg}^\vee_k$, there exists a regular semisimple element $h \in H$ such that $W^h \neq 0$ and $\Cent(\hG_k, h)$ is connected.
\end{enumerate}
\end{definition}
The roles of these two definitions are as follows: the strong irreducibility of $\overline{\rho}_\ffrm$ allows us to cut down $\cC_{G, k}^U$ to its finite rank $W(k)$-submodule $\cC_{G, k, \text{cusp}}^U$ using only Hecke operators (and not excursion operators). The $\hG$-abundance of $\overline{\rho}_\ffrm(\Gamma_{K(\zeta_\ell)})$ is used in the construction of sets of Taylor--Wiles places, which are the main input in the proof of Theorem \ref{thm_automorphy_lifting}. 

If $\ell$ is sufficiently large, then the group $\hG(\bbF_\ell)$ is both strongly $\hG$-irreducible (inside $\hG(\overline{\bbF}_\ell)$) and $\hG$-abundant (inside $\hG(k)$, for a sufficiently large finite extension $k / \bbF_\ell$). However, it is not clear how many other families of examples there are! This motivates the following question: 
\begin{question}
Can one prove an analogue of Theorem \ref{thm_automorphy_lifting} with weaker hypotheses? For example, can one replace conditions 3. and  4. with the single requirement that $\overline{\rho}_\ffrm$ is absolutely $\hG$-irreducible and $\ell$ is sufficiently large, relative to $\hG$?
\end{question}
To weaken the `$\hG$-abundant' condition is analogous to weakening the `bigness' condition which appeared in the first automorphy lifting theorems for unitary groups proved in \cite{Clo08}. It seems like an interesting problem to try, in a way analogous to \cite{Tho12}, to replace this condition with the $\hG$-irreducibility of the residual representation $\overline{\rho}_\ffrm$. 

\section{Coxeter parameters}

In order to apply a result like Theorem \ref{thm_automorphy_lifting}, we need to have a good supply of representations $\overline{\rho} : \Gamma_{K, \emptyset} \to \hG(\overline{\bbF}_\ell)$ which we know to be residually automorphic (in the sense of arising from a maximal ideal of the excursion algebra $\cB(U, W(k))$ acting on cuspidal automorphic forms). 

Famously, Wiles used the Langlands--Tunnell theorem to prove the residual automorphy of odd surjective homomorphisms $\overline{\rho} : \Gamma_\bbQ \to \GL_2(\bbF_3)$, in order to be able to use his automorphy lifting theorems to prove the modularity of elliptic curves. Many recent applications of automorphy lifting theorems (e.g. to potential automorphy of $n$-dimensional Galois representations over number fields, or to the construction of lifts of residual representations with prescribed properties, as in \cite{Bar14}) have relied upon the automorphy of $n$-dimensional Galois representations which are induced from a character of the Galois group of a cyclic extension of numbers fields of degree $n$. The automorphy of such representations was proved by Arthur--Clozel, using a comparison of twisted trace formulae \cite{Art89}.

We obtain residually automorphic Galois representations from a different source, namely the geometric Langlands program. We first describe the class of representations that we use. We fix a split maximal torus $\hT \subset \hG$, and write $W = W( \hG, \hT)$ for the Weyl group of $\hG$. We assume in this section that $\hG$ is simple and simply connected.
\begin{definition}
An element $w \in W$ is called a Coxeter element if it is conjugate to an element of the form $s_1 \dots s_r$, where $R = \{ \alpha_1, \dots, \alpha_r \} \subset \Phi(\hG, \hT)$ is any choice of ordered root basis and $s_1, \dots, s_r \in W$ are the corresponding simple reflections.
\end{definition}
It is a fact that the Coxeter elements form a single conjugacy class in $W$, and therefore have a common order $h$, which is called the Coxeter number of $\hG$. They were defined and studied by Coxeter in the setting of reflection groups. Kostant applied them to the study of reductive groups \cite{Kos59}, and his results form the foundation of our understanding of the following definition:
\begin{definition}
Let $\Gamma$ be a group, and let $\Omega$ be an algebraically closed field. We call a homomorphism $\phi : \Gamma \to \hG(\Omega)$ a Coxeter homomorphism if it satisfies the following conditions:
\begin{enumerate}
\item There exists a maximal torus $T \subset \hG_\Omega$ such that $\phi(\Gamma) \subset N(\hG_\Omega, T)$, and the image of $\phi(\Gamma)$ in $W \cong N(\hG_\Omega, T) / T$ is generated by a Coxeter element $w$. We write $\phi^\text{ad}$ for the composite of $\phi$ with projection $\hG(\Omega) \to \hG^\text{ad}(\Omega)$, and $T^\text{ad}$ for the image of $T$ in  $\hG^\text{ad}_\Omega$.
\item There exists a prime $t \equiv 1 \text{ mod }h$ not dividing $\text{char }\Omega$ or $\# W$ and a primitive $h^\text{th}$-root of unity $q \in \bbF_t^\times$ such that $\phi^\text{ad}(\Gamma) \cap T^\text{ad}(\Omega)$ is cyclic of order $t$, and conjugation by $w$ acts on the image by the map $v \mapsto v^q$.\footnote{As the notation suggests, in applications we will take $q$ to be the image in $\bbF_t$ of the cardinality of the field of scalars in $K$.}
\end{enumerate}
\end{definition}
We recall that if $\hG = \SL_n$, then $W = S_n$ and the Coxeter elements are the $n$-cycles. In this case the Coxeter homomorphisms appear among those homomorphisms $\Gamma \to \SL_n(\Omega)$ which are induced from a character of an index $n$ subgroup. However, the above definition is valid for any simply connected simple $\hG$ and has very good properties:
\begin{proposition}\label{prop_properties_of_coxeter_parameters}
Let $\phi : \Gamma \to \hG(\Omega)$ be a Coxeter homomorphism. Then:
\begin{enumerate}
\item $\phi$ is $\hG$-irreducible. 
\item If $\phi' : \Gamma \to \hG(\Omega)$ is another homomorphism such that for all $\gamma \in \Gamma$, $\phi(\gamma)^\text{ss}$ and $\phi(\gamma')^\text{ss}$ are $\hG(\Omega)$-conjugate, then $\phi$ and $\phi'$ are themselves $\hG(\Omega)$-conjugate. In particular, $\phi$ is even strongly $\hG$-irreducible.
\item The image $\phi(\Gamma)$ is an $\hG$-abundant subgroup of $\hG(\Omega)$.
\end{enumerate}
\end{proposition}
Now suppose that $\phi : \Gamma_{K, \emptyset} \to \hG(\overline{\bbQ}_\ell)$ is a Coxeter parameter. Then there exists a degree $h$ cyclic extension $K' / K$ such that $\phi(\Gamma_{K'})$ is contained in a torus of $\hG$; the homomorphism $\phi|_{\Gamma_{K'}}$ is  therefore associated to Eisenstein series on $G(\bbA_{K'})$. One can ask whether it is possible to use this to obtain a cuspidal automorphic representation of $G(\bbA_K)$ (or better, a maximal ideal of the excursion algebra $\cB_U$) to which $\phi$ corresponds. One case in which the answer is affirmative is as follows:
\begin{theorem}\label{thm_automorphy_of_coxeter_parameters}
Let $\phi : \Gamma_{K, \emptyset} \to \hG(\overline{\bbQ}_\ell)$ be a Coxeter parameter such that $\phi( \Gamma_{K \cdot \overline{\bbF}_q}) \subset \hT(\overline{\bbQ}_\ell)$. Then $\phi$ is cuspidal automorphic, in the sense of Definition \ref{def_cuspidal_automorphic_galois_representation}. 
\end{theorem}
\begin{proof}[Proof (sketch)]
Braverman--Gaitsgory construct \cite{Bra02} the geometric analogue of Eisenstein series for the group $G$: in other words, a Hecke eigensheaf on $\Bun_{G, \overline{\bbF}_q}$ with `eigenvalue' $\phi|_{\Gamma_{K \cdot \overline{\bbF}_q}}$. This Hecke eigensheaf is equipped with a Weil descent datum, which allows us to associate to it an actual spherical automorphic form $f : G(K) \backslash G(\bbA_K)  \to \overline{\bbQ}_\ell$ whose Hecke eigenvalues agree with those determined by $\phi$, under the Satake isomorphism. Using geometric techniques (see Gaitsgory's appendix to \cite{Boc17}), one can further show that this automorphic form is in fact cuspidal. The existence of a maximal ideal in the excursion algebra $\cB_U$ corresponding to $\phi$ then follows from the existence of $f$ and the good properties of Coxeter parameters (in particular, the second part of Proposition \ref{prop_properties_of_coxeter_parameters}).
\end{proof}
To illustrate the method, here is the result we obtain on combining Theorem \ref{thm_automorphy_of_coxeter_parameters} with our automorphy lifting Theorem \ref{thm_automorphy_lifting}:
\begin{theorem}\label{thm_automorphy_of_residually_coxeter_representation}
Let $\ell > \# W$ be a prime, and let $\rho : \Gamma_{K, \emptyset} \to \hG(\overline{\bbQ}_\ell)$ be a continuous homomorphism such that $\overline{\rho}$ is a Coxeter parameter and $\overline{\rho}(\Gamma_{K \cdot \overline{\bbF}_q})$ is contained in a conjugate of $\hT(\overline{\bbF}_\ell)$. Then $\rho$ is cuspidal automorphic, in the sense of Definition \ref{def_cuspidal_automorphic_galois_representation}.
\end{theorem}
Here is a question motivated by a potential strengthening of Theorem \ref{thm_automorphy_of_residually_coxeter_representation}:
\begin{question}\label{q_multiplicity_one_of_coxeter_parameters}
Let $\phi : \Gamma_{K, \emptyset} \to \hG(\overline{\bbQ}_\ell)$ be a Coxeter parameter such that $\phi( \Gamma_{K \cdot \overline{\bbF}_q}) \subset \hT(\overline{\bbQ}_\ell)$, and let $\pi$ be the everywhere unramified cuspidal automorphic representation of $G(\bbA_K)$ whose existence is asserted by Theorem \ref{thm_automorphy_of_coxeter_parameters}. Can one show that $\pi$ appears with multiplicity 1 in the space $\cA_{G, \text{cusp}}$?
\end{question}
 Taking into account the freeness assertion in Theorem \ref{thm_automorphy_lifting}, we see that a positive answer to Question \ref{q_multiplicity_one_of_coxeter_parameters} would have interesting consequences for the multiplicity of cuspidal automorphic representations.

\section{Potential automorphy}\label{sec_potential_automorphy}

We can now describe the proof of Theorem \ref{thm_BHKT}. Let us therefore choose a representation $\rho : \Gamma_{K, \emptyset} \to \hG(\overline{\bbQ}_\ell)$ of Zariski dense image. We must find a finite Galois extension $L / K$ such that $\rho|_{\Gamma_L}$ is cuspidal automorphic. It is easy to reduce to the case where $\hG$ is simple and simply connected (equivalently: the group $G$ is simple and has trivial centre), so we now assume this. 

By Proposition \ref{prop_Zariski_dense_image}, we can assume, after replacing  $\rho$ by a conjugate, that there is a number field $E$, a system $(\rho_\lambda)_\lambda$ of continuous homomorphisms $\rho_{\lambda} : \Gamma_{K, \emptyset} \to \hG(E_\lambda)$ of Zariski dense image, and an embedding $E_{\lambda_0} \hookrightarrow \overline{\bbQ}_\ell$ such that $\rho = \rho_{\lambda_0}$. If any one of the representations $\rho_\lambda$ is automorphic, then they all are. We can therefore forget the original prime $\ell$ and think of the entire system $(\rho_\lambda)_\lambda$.

An application of a theorem of Larsen \cite{Lar95} furnishes us with strong information about this system of representations:
\begin{theorem}\label{thm_larsen}
With notation as above, we can assume (after possibly enlarging $E$ and replacing each $\rho_\lambda$ by a conjugate), that there exists a set $\cL$ of rational primes of Dirichlet density 0 with the following property: for each prime $\ell \not\in \cL$ which splits in $E$ and does not divide $q$, and for each place $\lambda | \ell$ of $E$, $\rho_\lambda(\Gamma_{K, \emptyset})$ has image equal to $\hG(\bbZ_\ell)$.
\end{theorem}
It follows that for $\lambda | \ell$ ($\ell \not\in \cL$ split in $E$), the residual representation $\overline{\rho}_\lambda$ can be taken to have image equal to $\hG(\bbF_\ell)$. It is easy to show that when $\ell$ is sufficiently large, such a residual representation satisfies the requirements of our Theorem \ref{thm_automorphy_lifting}.

We are now on the home straight. Using a theorem of Moret-Bailly and known cases of de Jong's conjecture \cite{Mor90, deJ01}, one can construct a finite Galois extension $L / K$ and (after possibly enlarging $E$) an auxiliary system of representations $(R_\lambda : \Gamma_{L, \emptyset} \to \hG(E_\lambda))_\lambda$ satisfying the following properties:
\begin{enumerate}
\item For each prime-to-$q$ place $\lambda$ of $\overline{\bbQ}$, $R_\lambda$ has Zariski dense image in $\hG(\overline{\bbQ}_\lambda)$. 
\item There exists a place $\lambda_1$ such that $\overline{R}_{\lambda_1} \cong \overline{\rho}_{\lambda_1}|_{\Gamma_{L, \emptyset}}$, and both of these representations have image $\hG(\bbF_{\ell_1})$, where $\ell_1$ denotes the residue characteristic of $\lambda_1$. 
\item There exists a place $\lambda_2$ such that $\overline{R}_{\lambda_2}$ is a Coxeter parameter and $\ell_2 > \# W$, where $\ell_2$ denotes the residue characteristic of $\lambda_2$.
\end{enumerate}
The argument to prove the automorphy of $\rho|_{\Gamma_{L, \emptyset}}$ is now the familiar one. By Theorem \ref{thm_automorphy_of_residually_coxeter_representation}, $R_{\lambda_2}$ is cuspidal automorphic. Since this property moves in compatible systems for representations with Zariski dense image, $R_{\lambda_1}$ is cuspidal automorphic. If $\ell_1$ is chosen to be sufficiently large, then we can apply Theorem \ref{thm_automorphy_lifting} to deduce that $\rho_{\lambda_1}|_{\Gamma_{L, \emptyset}}$ is cuspidal automorphic. Moving now in the compatible system containing $\rho_{\lambda_1}|_{\Gamma_{L, \emptyset}}$, we obtain finally the automorphy of the original representation $\rho|_{\Gamma_{L, \emptyset}}$, as desired.

One of the main attractions of our arguments is that they are uniform in the reductive group $G$. In particular, they are valid for exceptional groups. Using deformation theory, it is easy to find examples of global fields $K = \bbF_q(X)$ and continuous representations $\rho : \Gamma_{K, \emptyset} \to \hG(\overline{\bbQ}_l)$ of Zariski dense image. This gives, for example, the following simple corollary of Theorem \ref{thm_BHKT}:
\begin{corollary}
Let $G$ be the split simple group over $\bbF_q$ of type $E_8$; then the dual group $\hG$ is the split simple group over $\bbZ$ of type $E_8$. Let $\ell$ be a prime not dividing $q$. Then there exist infinitely pairs $(K, \rho)$, where $K = \bbF_q(X)$ is a global field and $\rho : \Gamma_{K, \emptyset} \to E_8(\overline{\bbQ}_\ell)$ is a representation of Zariski dense image which is cuspidal automorphic, in the sense of Definition \ref{def_cuspidal_automorphic_galois_representation}.
\end{corollary}

\section{Acknowledgments}

I am very grateful to Chandrashekhar Khare, Michael Harris, and Vincent Lafforgue for their comments on a draft version of this article. This project has received funding
from the European Research Council (ERC) under the European Union's Horizon
2020 research and innovation programme (grant agreement No 714405).

\bibliographystyle{alpha}
\bibliography{icmbib}
\end{document}